\documentclass{ws-m3as}

\usepackage{epstopdf}% To incorporate .eps illustrations using PDFLaTeX, etc.
\usepackage[caption=false]{subfig}

\usepackage{xcolor}
\usepackage{float}

\newtheorem{defi}{Definition}
\newtheorem{prop}{Proposition}
\newtheorem{theo}{Theorem}
\newtheorem{lem}{Lemma}

\newcommand{\R}{\mathbb R}
\newcommand{\IR}{\mathbb R}

\newcommand{\IN}{\mathbb N}

\newcommand{\E}{\mathbb E}

\newcommand{\Pb}{\mathbb P}
\newcommand{\Laplace}{\boldsymbol{\Delta}}

\newcommand{\dif}{d \!}

\newcommand{\ch}{\mathcal H}

\newcommand{\cj}{\mathcal J}

\newcommand{\cn}{\mathcal N}

\begin{document}

\markboth{ Baltazar-Larios F., Delgado-Vences F., Peralta L.}{Statistical inference for a SPDE related to an ecological niche}

%%%%%%%%%%%%%%%%%%% Publisher's Area please ignore %%%%%%%%%%%%%%%%%%%%%%%
%
\catchline{}{}{}{}{}
%
%%%%%%%%%%%%%%%%%%%%%%%%%%%%%%%%%%%%%%%%%%%%%%%%%%%%%%%%%%%%%%%%%%%%%%%%%%

\title{Statistical inference for a stochastic partial differential equation related to an ecological niche}

\author{Fernando Baltazar-Larios}

\address{Facultad de Ciencias, UNAM, %Department and Organization
    Coyoac\'an, CDMX, M\'exico\\
fernandobaltazar@ciencias.unam.mx}

\author{Francisco Delgado-Vences\footnote{Corresponding author}}

\address{Conacyt-Instituto de Matematicas, UNAM, Oaxaca de Ju\'arez, M\'exico,\\
delgado@im.unam.mx}

\author{Liliana Peralta}

\address{Facultad de Ciencias, UNAM, %Department and Organization
    Coyoac\'an, CDMX, M\'exico\\
lylyaanaa@ciencias.unam.mx }

\maketitle

\begin{history}
\received{(Day Month Year)}
\revised{(Day Month Year)}
%\accepted{(Day Month Year)}
\comby{(xxxxxxxxxx)}
\end{history}

\begin{abstract}
In this paper,  we use a stochastic partial differential equation (SPDE) as a model for the density of a population under the influence of random external forces/stimuli given by the environment. We study statistical properties for two crucial parameters of the SPDE that describe the dynamic of the system. To do that we use the Galerkin projection to transform the problem, passing from the SPDE to a system of independent SDEs; in this manner, we are able to find the Maximum likelihood estimator of the parameters. We validate the method by using simulations of the SDEs. We prove consistency and asymptotic normality of the estimators; the latter is showed using the Malliavin-Stein method. We illustrate our results with numerical experiments.

\end{abstract}


\ccode{AMS Subject Classification: XXXXX}

\section{Introduction}

Since the foundation of the theory of deterministic biodiffusion has been noticed the necessity of more realistic models in ecology, for instance, a model that considers the interaction between animals, and stimuli from/to the environment (Ref.~\refcite{okubo2001diffusion}). Several attempts at random models have been tried to solve that; for instance, Ref.~\refcite{gloaguen2018stochastic} study a stochastic differential equation (SDE) whose drift
is the gradient of a multimodal potential surface. They used an Euler method to perform parameter inference and two other pseudo-likelihood procedures.  In Ref.~\refcite{dogan2016derivation}, the authors consider a biased and correlated random walk (BCRW) and they derived stochastic partial differential equations (SPDEs)  for a
BCRW in one, two, and three dimensions under the assumption that the globally preferred direction of movement is independent of the location of the particles. This assumption seems to be a limitation of their model. Further reading on stochastic modeling of animal movement can be found in Ref.~\refcite{smouse2010stochastic} for instance.

Our idea is to consider a stochastic equation inspired by a model proposed by Dipierro and Valdinoci, in Ref.~\refcite{Dipi-Pro} and Ref.~\refcite{Dipi-valdi}, where they study a deterministic model that describes the diffusion of a biological population living in an ecological niche subject to local and nonlocal dispersals. Their model considers a partial differential equation (PDE) with a mixture of classical and fractional Laplacians. Therefore, the diffusion follows two types of dispersals: a classical one, related to the usual Laplacian, and a nonlocal one, modeled on Lévy flights and encoded by the fractional Laplacian.

The proposed SPDE is a macroscopic model, in which we assume that the stimuli are only from the environment to the animals and there is no interaction between the animals. This SPDE could be appropriate to model birds' flights or sea animals' diffusion, for instance.

We are interested in estimating two parameters of the SPDE, however, doing that for the SPDE directly is a difficult task (see Ref.~\refcite{cia-18} for instance). By using the Galerkin projection we see that the coefficients of the spectral decomposition of the SPDE's solution are a system of independent Ornstein-Uhlenbeck (OU) processes, that we use to perform numerical approximations to the SPDE, and, to study statistical properties of some parameters in the SPDE.

Since we are assuming that the animals are not marking or modifying the territory, then it makes sense that the system of OU processes is independent. We want to mention the works of Ref.~\refcite{chavanis2010stochastic}, Ref.~\refcite{schweitzer1994clustering}, and Ref.~\refcite{stevens2000derivation}, in which the authors model the stimuli among the animals and between the animals and their environment.

Parameter estimation is very important in stochastic modeling. For an overview of statistical inference for SPDE, we refer to Ref.~\refcite{cia-18}. In particular, the investigation of maximum likelihood estimation for a certain class of SPDE is presented in Refs.~\refcite{hue-93} and \refcite{hue-95}. Bayesian estimation is considered in Ref.~\refcite{pra-00}, and estimation of a linear multiplier for SPDE is studied in Ref.~\refcite{rao-01}. A study of how parameter estimation techniques developed for simple linear SPDE models apply to cell repolarization was presented in Ref.~\refcite{Alt-22}. In Ref.~\refcite{hil-21},  studied parameter estimation for a parabolic linear stochastic partial differential equation in one space dimension when observing the solution field on a discrete grid in a fixed bounded domain. \\

Our work is related to the manuscript Ref.~\refcite{Loto2003}; in there, the author presents an estimation for a two-dimensional parameter from the observations of a random field defined on a compact manifold by a stochastic parabolic equation, in particular, the authors found conditions for the consistency and asymptotic normality of the estimator. However, the author considers only the case when the dimension of the projection increases, meaning the limit $T\rightarrow \infty$ is not studied there. In our work, we show the convergence when, both, $N$ (number of OU processes) and $T$ (time horizon)  go to infinity.

Since the SPDE considered in our model is linear and diagonalizable, we use the classical statistical method called the spectral approach (presented in Refs.~\refcite{hue-93} and \refcite{hue-95}) to obtain the maximum likelihood estimators (MLEs). We also prove the consistency and asymptotic normality of these parameters.\\

This paper is organized as follows. Section \ref{sec: models} contains previous works for modeling the distribution of biological species in ecological niches that are related to our model. In Section \ref{sec: our model} we describe the stochastic model used in this work. The calculation of the maximum likelihood estimators for the parameters of interest of our model is presented in Section \ref{sec: statistical},  and Section \ref{sec: con} contains the consistency properties of these parameters. In Section \ref{MC}   we prove the asymptotic normality of the estimators. We present a simulation study in Section \ref{simu}, where we calibrate and illustrate numerically the statistical methodology. Some concluding remarks are given in Section \ref{conclu}.\\

\section{Bibliographical comments}\label{sec: models} 

Several models and methods have been applied to ecological niches, such as support vector machines, generalized linear models, naive Bayes, artificial
neural networks, classification trees, etc. (see for instance Ref.~\refcite{guo-yu} and the references therein). Here we briefly review the most used method.

 One of the most important models for the distribution of biological species is the so-called MaxEnt\footnote{MaxEnt is the abbreviation of maximum entropy} method. This method is based on thermodynamic ideas and is in constant development (see for example Ref.~\refcite{aoki} and Ref.~\refcite{jor} for a general review of the model). It was introduced by Edwin T. Jaynes in Ref.~\refcite{jay-57-1} and Ref.~\refcite{jay-57-2} in the fifty's in the context of statistical mechanics and information theory, as an efficient and mathematically simple method to estimate probability distributions. This procedure was quickly adapted to studies of biological species in the context of ecology (for an overview of the development of MaxEnt see for instance Ref.~\refcite{har}).

The objective of the MaxEnt method, in ecology, is to estimate the potential distribution of a species according to the ``suitability" of the geographic environment that constitutes its natural habitat, under the fundamental thermodynamic principle that without external influences, the biological system tends naturally to a state of maximum entropy. The method solves the problem of determining the geographic distribution of species using presence-only data since it is based only on environmental conditions, which makes it a method capable of overcoming the deficiencies related to data quality (usually absence data are not available).
%%%%%%%%%%%%%%%%%%%%%%%%%%%%%%%%%%%%%%%%%%%%%%%%%%%%%%%%%%%%%%%%%%%%%%%%%%
%%%%%%%%%%%%%%%%%%%%%%%%%%%%%%%%%%%%%%%%%%%%%%%%%%%%%%%%%%%%%%%%%%%%%%%%%%
\section{Stochastic model}
\label{sec: our model}

Our proposal is inspired by the recent model of Dipierro et al.  (see Ref.~\refcite{Dipi-Pro} and Ref.~\refcite{Dipi-valdi}), where it presents a problem of population dynamics through a deterministic equation driven by a diffusive operator of mixed order, that is to say, the operator is the sum of a classical Laplacian $-\Laplace$ and fractional one $(-\Laplace)^s$, $s\in (0,1)$. This combination of Laplacians describes individuals who spread either by a random walk or by a jump process. For example, this can occur when exploring the environment and hunting at the same time.

Unlike the model presented in Ref.~\refcite{Dipi-valdi}, where the authors consider the integral definition of the fractional Laplacian (cf. Ref.~\refcite{se-va-14}), we will use its spectral definition, i.e., a different operator that is defined as a power of the classical Laplacian $-\Laplace$ with Neumann boundary conditions on a bounded and smooth enough domain $\mathcal{O}\subset \R^d$. 
We know (see Ref.~\refcite{gr-ng} or Section XIII.15 Ref.~\refcite{rs}) that:
\begin{itemize}
\item[a)] The set $\{h_k\}_{k\in\IN}$ of eigenfunctions of $-\Laplace$ forms a complete orthonormal system in $L^2(\mathcal{O})$.
\item[b)]\label{bp} The corresponding eigenvalues $\{\lambda_k\}_{k\in\IN}$, can be arranged such that $0=\lambda_1< \lambda_2\leq \lambda_3\leq \cdots$ with $\lambda_k\to +\infty$ as $k\to +\infty$.
%and we assume that this is a sequence of positive real numbers diverging
%to infinity.
\end{itemize}
In addition, it is also known that for domains with sufficiently regular boundaries there exists a positive constant $\varpi$ so that 
\begin{equation*}\label{eq:asymEigenv}
  \lim_{k\to\infty} |\lambda_k|k^{-2/d} = \varpi.
\end{equation*}
Then, the fractional Laplacian that we consider is the linear operator on $L^2(\mathcal{O})$ defined as  
\begin{align}\label{frac_laplace}
 u &= \sum_{k=1}^\infty u_k h_k \longmapsto (-\Laplace)^{s} u  =\sum_{k=1}^\infty \lambda^{s}_k u_k  h_k,
\end{align}
with domain 
\begin{equation*}\label{domain}
\mathcal{D}((-\Laplace)^s)=\left\{ u\in L^2(\mathcal{O})\mid u = \sum_{k=1}^\infty u_k h_k \text{ and } \sum_{k=1}^\infty u_k^2 \lambda_k^{2s} <+\infty \right\}.
\end{equation*}

Nonlocal operators such as integral fractional Laplacian have different types of applications. Particularly in ecology, a typical example is given by L\'evy flights, where the optimal search theory assumes that predators should adopt search strategies based on long jumps (see for instance Ref.~\refcite{rey}). In addition, we remark that operators defined with the spectral decomposition, as in the case of the one established in \eqref{frac_laplace}, also have applications in biology, for instance, see Ref.~\refcite{monte} the authors prove that the long jumps random walks are generated by this spectral operator with Neumann boundary conditions. On the other hand, we have chosen $(-\Laplace)^{\frac{1}{2}}$ in such a manner that commutes with $-\Laplace$ and that the SPDE, which we shall define next, can be diagonalized; this is a convenient property that we exploited heavily.\\

Now we present the SPDE that we will analyze in the rest of the work. Consider a stochastic basis  $(\Omega,\mathcal{F},\{\mathcal{F}_t\}_{t\geq 0}, \Pb)$ that satisfies the usual assumptions, that is, $(\Omega,\mathcal{F}, \Pb)$ is a complete probability space and $\{\mathcal{F}_t\}_{t\geq 0}$ is a filtration satisfying the usual hypothesis. On this basis, we denote by  $\{w_j,\ j\geq 1\}$ a collection of independent standard Brownian motions. For $d=2$ we consider the following system
\begin{align}\label{stochastic_local_nonlocal}
\begin{cases}
\displaystyle du-\theta \Laplace u\, dt +\beta (-\Laplace)^{\frac{1}{2}}u\, dt = \sigma \sum_{k\in\IN} \lambda_k^{-\gamma}h_k(x)\, d w_k(t),\; &\text{in } (0,+\infty)\times\mathcal{O},\\
\frac{ \partial u}{\partial\eta } =0,\;  &\text{on }(0,+\infty)\times \partial \mathcal{O},\\
u(0,x)= U_0(x), \; &\text{in } \mathcal{O}, 
\end{cases}
\end{align}
where $\theta>0, \beta>0$ are parameters to be estimated. In addition, $\ \gamma \geq 0$, $\sigma>0$, and $U_0\in L^2(\mathcal{O})$.\\

\begin{remark}

{\it Observe that in the defintion of the fractional Laplacian \eqref{frac_laplace}, the parameter $s$ belongs to the interval $(0,1)$; however, in the SPDE \eqref{stochastic_local_nonlocal} we have taken $s=1/2$. It is important to remark that all the results of this paper remain valid for the case $s\in (0,1)$,  although the proofs has to be adjusted accordingly.}
    
\end{remark}
 We could interpret $u=u(t,x)$, in equation \eqref{stochastic_local_nonlocal}, as the density of a population, living in a biological niche represented by $\mathcal{O}$, under the influence of drag forces and random forces. We can see that equation \eqref{stochastic_local_nonlocal} is within the general setting of Section 4.4.1 Ref.~\refcite{LototskyRozovsky2017Book}, therefore it has a unique adapted strong solution $u$ in probability sense, such that
\begin{equation*}%\label{eq:existuniqueSol}
u\in  L^{2}(\Omega; C((0,T); L^2(\mathcal{O}))).
\end{equation*}

Let us denote by $H^N$ the finite-dimensional subspace of $L^2(\mathcal{O})$ generated by $\{h_k\}_{k=1}^N$, and denote by $P_N$ the projection operator of $L^2(\mathcal{O})$ into $H^N$ and put $U^N = P_N\, u$. Moreover, let $u_k,k\in\IN,$ be the Fourier coefficient of the solution $u$ of \eqref{stochastic_local_nonlocal} with respect to $h_k,k\in\IN$, i.e. $u_k(t) = \big(u(t,\cdot),h_k(\cdot)\big)_{L^2(\mathcal{O})}, k\in\IN$. It is not difficult to see that the Fourier coefficients $u_k,k\in\IN$ follow the dynamics of an OU process in $\IR$ given by
\begin{equation}\label{eq:OU-Fourier}
\dif u_k = \big(-\theta \lambda_k - \beta \lambda_k^{1/2}\big)  u_k dt + \sigma \lambda_k^{-\gamma} d w_k(t), \quad  u_k(0) = (U_0,h_k), \ t\geq 0.
\end{equation}

Observe that the coefficients of the SDE \eqref{eq:OU-Fourier} are scalar, simplifying the calculations of the expectation and the second moment. In addition, we have a strong solution for each $u_k$ given by (see Ref.~\refcite{iacus2008simulation}, Ref.~\refcite{oksendal2013stochastic}
for instance)
\begin{equation}\label{uk_solution}
 u_k(t)  =  u_k(0) e^{(-\theta\lambda_k -\beta\lambda_k^{1/2} )t} + \sigma \lambda_k^{-\gamma}\int_0^t e^{(-\theta\lambda_k -\beta\lambda_k^{1/2} )(t-s)} dw_k(s).
\end{equation}

It is not difficult to get the first two moments:
\begin{eqnarray}\label{2nd-moment-uk}
\E\big(u_k(t) \big) & = & u_k(0) e^{-(\theta\lambda_k + \beta\lambda_k^{1/2} )t},\nonumber\\ 
 \E\big(u_k^2(t) \big) &= &\frac{\sigma^2 \lambda_k^{-2\gamma}}{2(\theta\lambda_k + \beta \lambda_k^{1/2})} \Big[ 1- e^{-2(\theta\lambda_k +\beta\lambda_k^{1/2} )t} \Big].
\end{eqnarray}

It is well-known that we can write the solution $u$ of the SPDE (with Neuman conditions) \eqref{stochastic_local_nonlocal} as
\begin{align} \label{u_spectral}
u(t,x)= \sum_{k=0}^\infty u_k(t)h_k(x) = 1 + \sum_{k=1}^\infty u_k(t)h_k(x),   
\end{align}
and we can approximate $u$ (in $L^2(\mathcal{O})$) with 
\begin{align}\label{u_truncation}
U^N(t,x):= \sum_{k=0}^N u_k(t)h_k(x) = 1 + \sum_{k=1}^N u_k(t)h_k(x).   
\end{align}
Note that $U^N$ is a finite sum of OU processes multiplied by the elements of the basis $h_k(\cdot) $; that means we could approximate the solution of the SPDE  \eqref{stochastic_local_nonlocal}.
 
\section{Statistical Inference }\label{sec: statistical}

Here, we will adapt the methods presented in Ref.~\refcite{CDK}. We study the MLEs for the parameters of interest, $\theta$ and $\beta$, of the SPDE \eqref{stochastic_local_nonlocal}. First, we introduce some notation that we shall use for the rest of the work.\\

\textbf{\textsc{Notation}}: For $p=\frac{1}{2}, 1, \frac{3}{2}, 2$ and $q=1,2$ we set
\begin{align*}
    I_{p,q} & := \sum_{k=1}^N   \lambda_k^{p+q\gamma} \int_0^T u_k^2(t) dt,\\
    J_{p,q} & := \sum_{k=1}^N   \lambda_k^{p+q\gamma} \int_0^T u_k(t) du_k(t),\\
    \cj_{p,q} & := \sum_{k=1}^N   \lambda_k^{p+q\gamma} \int_0^T u_k(t) dw_k(t).
\end{align*}

For two sequences of real numbers $\{a_n\}$ and $\{b_n\}$, we write $a_n \sim b_n$, if there exists  a constant $0<c$  such that $\lim_{n\to\infty}a_n/b_n=c$, and $a_n\simeq b_n$, if $\lim_{n\to\infty}a_n/b_n=1$.\\

\subsection{MLEs for $\theta$ and $\beta$}
Let $\Pb^{N,T}_{\theta, \beta}$ be the probability measures on $C([0,T]; H^N)\backsimeq C([0,T]; \R^N)$ generated by the $U^N$, which are equivalent for different values of $\theta$ and $\beta$. In addition, the Radon-Nikodym derivative or likelihood ratio has the form
\begin{align*}
\frac{\Pb^{N,T}_{\theta, \beta}}{\Pb^{N,T}_{\theta_{0}, \beta_0}} (U^N) &=
\exp\Bigg(\sum_{k=1}^N \frac{\left(\big[-\theta \lambda_k- \beta \lambda_k^{1/2}\big]- \big[-\theta_0 \lambda_k- \beta_0 \lambda_k^{1/2}\big] \right)}{\sigma^{2} \lambda_k^{-2\gamma}} \int_0^T  u_k(t)d u_k(t)\\
&\qquad - \sum_{k=1}^N \frac{\left(\big[-\theta \lambda_k - \beta \lambda_k^{1/2}\big]^2- \big[-\theta_0 \lambda_k - \beta_0 \lambda_k^{1/2}\big]^2 \right)}{2\sigma^{2} \lambda_k^{-2\gamma}}\int_0^Tu_k^2(t)d t\Bigg). %\label{eq:RadonNikodymUn}
\end{align*}

Maximizing the log-likelihood ratio with respect  to $\theta$ and $\beta$,  we obtain the MLEs
\begin{align}
\widehat{\theta}_{N,T} & := \frac{1}{I_{2,2}} \Bigg[ -\widehat{\beta}_{N,T}\, \sum_{k=1}^{N}\lambda_k^{3/2 +2\gamma}\int_0^T u_k^2(t)d t -\sum_{k=1}^{N}\lambda_k^{1 +2\gamma}\int_0^T u_k(t)d u_k(t)  \Bigg]\nonumber\\
&= \frac{1}{I_{2,2}} \Big[ -\widehat{\beta}_{N,T}\,  I_{3/2,2} - J_{1,2} \Big], \label{eq:MLE-theta}\\
\widehat{\beta}_{N,T} & := \frac{1}{I_{1,2}} \Bigg[ -\widehat{\theta}_{N,T}\, \sum_{k=1}^{N}\lambda_k^{3/2 +2\gamma}\int_0^T u_k^2(t)d t -\sum_{k=1}^{N}\lambda_k^{1/2 +2\gamma}\int_0^T u_k(t)d u_k(t)  \Bigg]\nonumber\\
&= \frac{1}{I_{1,2}} \Big[ -\widehat{\theta}_{N,T}\,  I_{3/2,2} - J_{1/2,2} \Big], \label{eq:MLE-beta}
\end{align}
with $ N\in\IN, \ T>0$.

By substituting the OU process, meaning the SDE \eqref{uk_solution}, we rewrite the last two expressions as a system of equations

\begin{equation*}
    \label{MLE-stochastic_version1}
    \begin{pmatrix}
        I_{2,2}\; &\; I_{3/2,2} \\
        I_{3/2,2} \; & \; I_{1,2}
    \end{pmatrix}
    \begin{pmatrix}
       \widehat{\theta}_{N,T}  -\theta_0  
        \\
        \widehat{\beta}_{N,T} - \beta_0
    \end{pmatrix}
    =
        - \sigma \begin{pmatrix} 
        \, \cj_{1,1}    \\
        \cj_{1/2,1}
    \end{pmatrix},
\end{equation*}

therefore

 \begin{align}
    \label{MLE-stochastic_version}
    \begin{pmatrix}
       \widehat{\theta}_{N,T}  -\theta_0  
        \\
        \widehat{\beta}_{N,T} - \beta_0
    \end{pmatrix}
& = \frac{-\sigma}{I_{2,2}I_{1,2} - I_{3/2,2}^2 }     \begin{pmatrix}
        I_{1,2}\; &\; -I_{3/2,2} \\
        -I_{3/2,2}\; & \;I_{2,2}
    \end{pmatrix}
    \begin{pmatrix} 
         \cj_{1,1}    \\
        \cj_{1/2,1}
    \end{pmatrix},
\end{align}
where the quantity $I_{4,2}I_{2,2} - I_{3,2}^2$ is strictly positive (see inequality \eqref{positive_quotient} below), and thus the expression  \eqref{MLE-stochastic_version} is well-defined.

%We now prove the main result of this section. 

%%%%%%%%%%%%%%%%%%%%%%%%%%%%%%%%%%%%%%%%%%%%%%%%%%%%%%%%%%%%%%%%%%
\subsection{Fisher's information}
Now we will compute the Fisher's information (FI) related to $\Pb^{N,T}_{\theta,  \beta}/\Pb^{N,T}_{\theta_{0}, \beta_0}$. For simplicity, set $U_0=0$. We recall that for both parameters we have that the FI is given by a symmetric $(2\times 2)$-matrix. Furthermore, if we fixed one of the parameters, we obtain that the FI is a scalar. Assuming that $\beta$ is fixed, we get
\begin{align*}%\label{eq:FisherInfo}
  \mathcal{I}_{N,T;\theta} & := \int \left|  \frac{\partial}{\partial \theta} \log \frac{\Pb^{N,T}_{\theta, \beta}}{\Pb^{N,T}_{\theta_0, \beta_0}} \right|^2 d \Pb^{N,T}_{\theta_0, \beta_0} \nonumber\\
 & = - \int  \frac{\partial^2}{\partial \theta^2} \left(\log \frac{\Pb^{N,T}_{\theta, \beta}}{ \Pb^{N,T}_{\theta_0, \beta_0}}\right)
 d \Pb^{N,T}_{\theta_0, \beta_0}  \nonumber\\
  & =  \frac{1}{\sigma^2} \sum_{k=1}^{N} \lambda_k^{2 +2\gamma }\E\left[ \int_0^T u_k^2(t) d t \right].
\end{align*}

 The moments computed in \eqref{2nd-moment-uk} imply
$$
\E\left[ \int_0^T u_k^2 d t \right] = \frac{\sigma^2 \lambda_k^{-2\gamma}}{2(\theta_0\lambda_k + \beta_0 \lambda_k^{1/2})} \left(T+\frac{1-e^{(-2\theta_0\lambda_k -2\beta_0\lambda_k^{1/2} )T}}{(-2\theta_0\lambda_k -2\beta_0\lambda_k^{1/2} )}\right),
$$
therefore, without loss of generality, we can assume that $\lambda_k\geq 1$ and to get 
\begin{align}
 \mathcal{I}_{N,T;\theta_0} & = \sum_{k=1}^{N}\frac{ \lambda_k^2}{2(\theta_0\lambda_k + \beta_0 \lambda_k^{1/2})} \left(T-\frac{1-e^{-2(\theta_0\lambda_k+\beta_0\lambda_k^{1/2})T}}{2(\theta_0\lambda_k +\beta_0\lambda_k^{1/2})}\right)\nonumber\\
& \sim T\sum_{k=1}^N \frac{\lambda_k^{2}}{(\theta_0\lambda_k + \beta_0 \lambda_k^{1/2})} ,\quad \mbox{as}\ T\to \infty
\nonumber\\
& > T\sum_{k=1}^N \frac{\lambda_k^{2}}{(\theta_0+\beta_0)\lambda_k } = \frac{T}{(\theta_0+\beta_0)}\sum_{k=1}^N \lambda_k,\quad \mbox{as}\ T\to \infty
\nonumber\\
&  \sim \frac{\varpi dTN^{\frac{2}{d}+1}}{\left(d+2\right)\, (\theta_0+\beta_0)},\quad \mbox{as}\ N,T\to \infty. 
\label{eq:FisherInfo2}
\end{align}
                                                            In particular, note that $ \mathcal{I}_{N,T;\theta_0}\to\infty$, when $N,T\to\infty$. In very similar manner, we can see that $ \mathcal{I}_{N,T;\beta_0}\to\infty$, when $N,T\to\infty$.

We now calculate the FI for both parameters. Set
\begin{equation*}
\varphi:=\varphi_{\theta,\beta}^{N,T} := \log \frac{\Pb^{N,T}_{\theta, \beta}}{\Pb^{N,T}_{\theta_{0}, \beta_0}} (U^N).
\end{equation*}
Thus
\begin{align*}
\mathcal{I}_{N,T;\theta,\beta}:=    \begin{pmatrix}
         -\int  \frac{\partial^2 \varphi}{\partial \theta^2} \, 
 d \Pb^{N,T}_{\theta_0, \beta_0}  & -\int  \frac{\partial^2 \varphi}{\partial \beta \partial\theta} \, 
 d \Pb^{N,T}_{\theta_0, \beta_0}  \\
 &
 \\
    -\int  \frac{\partial^2 \varphi}{\partial \theta \partial\beta} \, 
 d \Pb^{N,T}_{\theta_0, \beta_0}   & -\int  \frac{\partial^2 \varphi}{\partial \beta^2} \, 
 d \Pb^{N,T}_{\theta_0, \beta_0} 
    \end{pmatrix}=:  \begin{pmatrix} 
    A_{1,1}& A_{1,2}\\
    & \\
    A_{2,1}& A_{2,2}
    \end{pmatrix}.  
\end{align*}

We have already calculated the entry $A_{1,1}$. We calculate now $A_{1,2}=A_{2,1}$
$$
A_{1,2}= \frac{1}{\sigma^2} \sum_{k=1}^{N} \lambda_k^{3/2 +2\gamma} \E\left[ \int_0^T u_k^2(t) d t \right].
$$

and $A_{2,2}$
$$
A_{2,2}= \frac{1}{\sigma^2} \sum_{k=1}^{N} \lambda_k^{1 +2\gamma} \E\left[ \int_0^T u_k^2(t) d t \right].
$$
In analogy with the calculations to obtain the inequality \eqref{eq:FisherInfo2} we could see that 

\begin{align*}
det\big(\mathcal{I}_{N,T;\theta,\beta}\big) \sim C \, T\, N^{\frac{2}{d}+1}
\end{align*}
where $C:=C(\theta_0.\beta_0, \varpi, d)$ is a constant.

\section{Consistency of the estimators}\label{sec: con}

Now we will study the consistent of the parameters $\theta$ and $\beta$ in equation \eqref{stochastic_local_nonlocal}. We have the following result.

\begin{theorem}\label{theo:cons}
 The estimators $(\widehat{\theta}_{N,T}, \widehat{\beta}_{N,T} )$, given by  \eqref{eq:MLE-theta} and \eqref{eq:MLE-beta}, are strongly consistent, that is 
  \begin{equation*}
   \lim_{(N,T) \rightarrow \infty} \begin{pmatrix}
       \widehat{\theta}_{N,T}    
        \\
        \widehat{\beta}_{N,T} 
    \end{pmatrix} = \begin{pmatrix}
       \theta_0  
        \\
       \beta_0
    \end{pmatrix}.
  \end{equation*}
  
\end{theorem}
\begin{proof}

First,  we decompose  \eqref{MLE-stochastic_version} in two terms as follows: 

\begin{align}
     \widehat{\theta}_{N,T}  -\theta_0  
    &= \sigma \Bigg[ \frac{-I_{1,2} \cj_{1,1} } {I_{2,2}\,I_{1,2}- I_{3/2,2}^2 } + \frac{I_{3/2,2} \cj_{1/2,1} }{I_{2,2}\,I_{1,2} - I_{3/2,2}^2 } \Bigg]\nonumber\\
&=: \sigma\big( -T_1 + T_2\big).\label{decomp1}
    \end{align}

We continue the proof by studying the terms $T_1$ and $T_2$ separately.
We rewrite the quantity $T_1$ as follows

\begin{align}\label{T1_def}
 T_1 &=  \frac{ \, \sum_{k=1}^N Var(\xi_{k,T}) } {I_{2,2} - \frac{I_{3/2,2}^2}{I_{1,2}} } \frac{\sum_{k=1}^N \xi_{k,T}}{\sum_{k=1}^N Var(\xi_{k,T})} ,
\end{align}
where
\begin{equation*}
\xi_{k,T} :=\lambda_k^{1+\gamma}
\int_0^T u_k(t) dw_k(t).
\end{equation*}

We will prove that the second factor on the right side of \eqref{T1_def} converges to zero in probability, while the first one converges to some finite number $\rho^*$, also in probability. Applying the same arguments as in the last part of the proof of Theorem 1 in Ref.~\refcite{CDK} we have that

%We use the Theorem \ref{SLLN} {\color{magenta} No encontre este teorema en la referencia que dan, hay un teorema de ley fuerte de los grande n\'umeros que no me queda claro que sea el que han puesto en el ap\'endice}to prove the convergence to zero. 
%
%It is clear that $\E(\xi_{k,T})=0$ and 
%$$
%Var(\xi_{k,T})=\E(\xi_{k,T}^2)= \lambda_k^{{\color{red}2+2\gamma}}\int_0^T \E\big(u_k^2(t)\big) dt >0.
%$$
%Furthermore, using \eqref{2nd-moment-uk} and \eqref{eq:asymEigenv} we have that
%{\color{magenta} $$
%\lim_{k,T\rightarrow \infty } \E(\xi_{k,T}^2) = C<\infty.$$ ¿Tienen la cuenta donde muestran que este límite es finito? ¿ por que $T\to \infty$?}
%
%Thus, for the first part of Theorem \ref{SLLN}, we conclude that 

\begin{align}\label{T11_convergence}
  \frac{\sum_{k=1}^N \xi_{k,T}}{\sum_{k=1}^N Var(\xi_{k,T})} \longrightarrow 0 \mbox{ in probability when }\, N,T \rightarrow\infty.  
\end{align}

Focus on the inverse of the second factor in \eqref{T1_def}. Observe that we can have

\begin{align*}
0<\frac{I_{2,2} - \tfrac{I_{3/2,2}^2}{I_{1,2}}}{  \sum_{k=1}^N Var(\xi_{k,T}) } < \frac{I_{2,2}}{  \sum_{k=1}^N Var(\xi_{k,T}) }. 
\end{align*}
% thus, 
% \begin{align*}
% 0<\frac{\sum_{k=1}^N Var(\xi_{k,T})}{I_1(1-J_1 J_2)} \le \frac{\sum_{k=1}^N Var(\xi_{k,T})}{ \sum_{k=1}^N \lambda_k^{4+2\gamma}
% \int_0^T u_k^2(t) dt}
% \end{align*}
Furthermore, it is not difficult to see that 

$$
\frac{I_{2,2}}{  \sum_{k=1}^N Var(\xi_{k,T})} 
\longrightarrow 1 \mbox{ in   $L^1$-norm when }\, N,T \rightarrow\infty,
$$
and using similar arguments as in Ref.~\refcite{CDK} we have that

% $
% \E(\xi_{k,T}^4)\le C \big[\E(\xi_{k,T}^2)\big]^2$
% then, for the second part of Theorem \ref{SLLN} {\color{magenta} la referencia de este teorema no coincide con la que indican, sobre todo en la hip\'otesis donde sale el $\alpha$} 
%
\begin{align*}
\frac{I_{2,2}  }{  \sum_{k=1}^N Var(\xi_{k,T}) } \longrightarrow 1 \mbox{ in probability when }\, N,T \rightarrow\infty, 
\end{align*}
and from this, we deduce that 
\begin{align*}
0<\frac{I_{2,2} - \frac{I_{3/2,2}^2}{I_{1,2}}}{  \sum_{k=1}^N Var(\xi_{k,T}) } \longrightarrow \rho \mbox{ in probability when }\, N,T \rightarrow\infty,  
\end{align*}
for some $\rho\in (0,1)$, which implies that 
\begin{align*}
\frac{  \sum_{k=1}^N Var(\xi_{k,T}) }{I_{2,2} - \frac{I_{3/2,2}^2}{I_{1,2}}}\longrightarrow 1/\rho\, \mbox{ in probability when } N,T \rightarrow\infty.  
\end{align*}
This shows that $T_1 \longrightarrow 0$  in probability when $ N,T \rightarrow\infty$.

For $T_2$, define 
$$
\eta_{k,T}:= \lambda_k^{1/2+\gamma} \int_0^T u_k(t) dw_k(t) 
$$
then, $\E(\eta_{k,T})=0$ and 
$$
Var(\eta_{k,T})= \E(\eta_{k,T}^2)=  \lambda_k^{1+2\gamma} \int_0^T \E(u_k^2(t)) dt.
$$
Then, we rewrite $T_2$ as
\begin{align*}
 T_2&= \frac{I_{3/2,2}\,\sum_{k=1}^N Var(\eta_{k,T}) }{I_{2,2}\,I_{1,2} - I_{3/2,2}^2 } \frac{\sum_{k=1}^N \eta_{k,T}}{\sum_{k=1}^N Var(\eta_{k,T})}.
\end{align*}

For the second factor, using similar arguments as we have used to obtain \eqref{T11_convergence} we conclude that  
\begin{align*}%\label{T21_convergence}
  \frac{\sum_{k=1}^N \eta_{k,T}}{\sum_{k=1}^N Var(\eta_{k,T})} \longrightarrow 0 \mbox{ in probability when }\, N,T \rightarrow\infty.  
\end{align*}

To prove that the first factor in $T_2$ converges to a finite number we define 
$$
f:=f_N(t):=\sum_{k=1}^N \lambda_k^{1/2+\gamma} u_k(t),\quad 
g:=g_N(t):=\sum_{k=1}^N \lambda_k^{1+\gamma} u_k(t).
$$
Then, it is clear that $f,g$ belong to the  Hilbert Space  $\mathcal{H}:=L^2([0,T]\times \mathcal{M})$ where $\mathcal{M}$ is the counting measure on $\mathbb{N}$\,\footnote{See Section \ref{MC} below, for the definition of the space $\mathcal{H}$.}. In addition, it is not difficult to see that $f,g$ are linearly independent on $\mathcal{H}$. Now we write the inner product on $\mathcal{H}$ of $f,g$:
\begin{align*}
    \langle  f,g \rangle_{\mathcal{H}} &= \sum_{k=1}^N \int_0^T \Big[ \lambda_k^{1/2+\gamma} u_k(t) \Big] \Big[ \lambda_k^{1+\gamma} u_k(t) \Big] dt
= \sum_{k=1}^N \int_0^T  \lambda_k^{3/2+2\gamma} u_k^2(t)dt
= I_{3/2,2},
\end{align*}
thus, by applying the strict Cauchy-Schwarz's inequality (since $f,g$ are linearly independent in $\ch$) we have
\begin{align}\label{positive_quotient}
0<I_{3/2,2}^2 &=  \langle  f,g \rangle_{\mathcal{H}}^2 \, <\, \| f \|_\mathcal{H}^2 \|g \|_\mathcal{H}^2 \nonumber \\
& = \Bigg(\sum_{k=1}^N \int_0^T  \lambda_k^{1+2\gamma} u_k^2(t) \, dt\Bigg) \, \Bigg( \sum_{k=1}^N \int_0^T  \lambda_k^{2+2\gamma} u_k^2(t)\, dt \Bigg) \nonumber \\
&= I_{1,2} \, I_{2,2}.
\end{align}
Then, there exists a sequence of functions $\{\alpha_N(T)\}$ such that $0<\alpha_{N}(T)<1$ for all $N,T$ and that satisfies 
\begin{equation}\label{alpha_def}
    I_{3/2,2}^2 = \alpha_N(T) \, I_{2,2} \, I_{1,2}.
\end{equation}

Thus, we can write the first term in $T_2$ as follows
\begin{align*}
\frac{I_{3/2,2}\,\sum_{k=1}^N Var(\eta_{k,T}) }{I_{2,2}\,I_{1,2} - I_{3/2,2}^2 } &= \frac{I_{3/2,2} }{I_{2,2} \, I_{1,2} - \alpha_{N}(T) I_{2,2} \, I_{1,2}}\,\, \sum_{k=1}^N Var(\eta_{k,T})\nonumber\\
&= \frac{I_{3/2,2} }{(1-\alpha_N(T))\, I_{2,2} } \frac{\,\sum_{k=1}^N Var(\eta_{k,T})}{I_{1,2}}
\end{align*}

We rewrite the term 
$$
\frac{I_{3/2,2} }{ I_{2,2} } =  \frac{I_{3/2,2} }{ \E(I_{2,2}) }\, \frac{ \E(I_{2,2}) }{ I_{2,2} } =: Z_1 \, Z_2
$$
Then, using property b) of the sequence $\{\lambda_k\}$ we deduce that $Z_1$ converges to a finite constant in $L^2$-norm, and thus in probability.  

For $Z_2$,  we deduce that $1/Z_2$ converges to one in  $L^2$-norm, and thus in probability. Therefore,  $Z_2$ converges to one in probability. This implies that

$$
\frac{I_{3/2,2} }{ I_{2,2} }  \longrightarrow c_1<\infty \, \mbox{ in probability when }\, N,T \rightarrow\infty.
$$

Using the same arguments as in the last part of the proof of Ref.~\refcite{CDK}[Theorem 1] we can prove that
\begin{align*}
 \frac{\sum_{k=1}^N Var(\eta_{k,T})}{I_{1,2}} \longrightarrow 1 \, \mbox{ in probability when }\, N,T \rightarrow\infty.
\end{align*}

From these, and since $1/(1-\alpha_N(T))$ is bounded for all $N$ and $T$, we have that the first term in $T_2$ converges to 0 in probability.
 
Joining these last expressions and using Slutsky's theorem  (see Ref.~\refcite{casella-berger}, in particular, Th. 5.5.17 therein) we obtain
$T_2 \longrightarrow 0$  in probability when $ N,T \rightarrow\infty$. \\

This shows that $\theta_{N,T} \longrightarrow \theta_0$  in probability when $ N,T \rightarrow\infty$.\\

Similar arguments show that $\beta_{N,T} \longrightarrow \beta_0$  in probability when $ N,T \rightarrow\infty$. This concludes the proof.
\end{proof}

%%%%%%%%%%%%%%%%%%%%%%%%%%%%%%%%%%%%%%%%%%%%%%%%%%%%%%%%%%%%%%%%%%%%%%%%%%%%%%%%%%%%%%%%%%%%%%%%%%%%%%%%%%%%%%%%%%%%%%%%%%%%%%%%%%%%%%%%%%%%%%%%%%%%%%%%%%%%%%%%%%%%%%%%%%%%%%%%%%%%%%%%%%%%%%%%%%%%%%%%%%%%%%%%%%%%%%%%%%%%%%
\section{Asymptotic normality through Malliavin-Stein Method}\label{MC}
This section is devoted to proving the asymptotic normality. First, we enunciate the main result of the section; afterward, for the sake of completeness, we review briefly the Malliavin-Stein method in the appendix.

\begin{theorem}\label{theo:CL}
The estimators $(\widehat{\theta}_{N,T}, \widehat{\beta}_{N,T} )$, given by  \eqref{eq:MLE-theta} and \eqref{eq:MLE-beta}, are asymptotically normal,  that is,
  \begin{align*}
   w-\lim_{(N,T) \rightarrow \infty} 
        \sqrt{T} N^{1/d\, + 1/2 } \big( \widehat{\theta}_{N,T}-  \theta_0 \big)  &=  \cn (0, C_1) \\
         w-\lim_{(N,T) \rightarrow \infty} 
        \sqrt{T} N^{1/d\, + 1/2 } \big( \widehat{\beta}_{N,T}-  \beta_0 \big)  &=  \cn (0, C_2) 
  \end{align*}
  
  where $C_1$ and $C_2$ are constants not depending from $T,N$.
\end{theorem}

Observe that from equation \eqref{decomp1} and using $\alpha_N(T)$ as we have defined in \eqref{alpha_def}, we have the following expression
\begin{align*}
    \sqrt{T} N^{1/d\, + 1/2 } \, \big(   \widehat{\theta}_{N,T}  -\theta_0 \big)  & = \frac{1}{1-\alpha_N(T)} \sqrt{T} N^{1/d\, + 1/2 } \, \frac{\sqrt{var(\cj_{1,1})}}{I_{2,2}} \frac{\cj_{1,1}}{\sqrt{var(\cj_{1,1})}}\nonumber\\
    & \, + \frac{\alpha_N(T)}{1-\alpha_N(T)} \sqrt{T} N^{1/d\, + 1/2 } \, \frac{\sqrt{var(\cj_{1/2,1})}}{I_{3/2,2}} \frac{\cj_{1/2,1}}{\sqrt{var(\cj_{1/2,1})}}.
\end{align*}

We will prove, using the Malliavin-Stein method, that for $p=1, 1/2$
$$
M_{N,T}:=\frac{\cj_{p,1}}{\sqrt{var(\cj_{p,1})}}
$$
converges asymptotically to a Gaussian distribution, when $N,\, T \,\rightarrow\, \infty$. Moreover, we have the following convergences.

\begin{lem}
The expressions 
$$
\sqrt{T} N^{1/d\, + 1/2 } \, \frac{\sqrt{var(\cj_{1,1})}}{I_{2,2}} \qquad
   \mbox{ and } \qquad
 \sqrt{T} N^{1/d\, + 1/2 } \,  \frac{\sqrt{var(\cj_{1/2,1})}}{I_{3/2,2}}  $$
converge in probability to some finite constants, $C_1$ and $C_2$, when both $N,T \rightarrow \infty$. Therefore, converges in distribution to the same constants. 
\end{lem}

From these two convergences and using Slutsky's theorem, we shall obtain the asymptotic normality.  

Now we prove the lemma.
\begin{proof}
Set $p=3/2,2$. By following the procedure to obtain the estimation \eqref{eq:FisherInfo2},  we arrive to 
\begin{align}\label{tech_1}
  \E\big( I_{p,2}\big) &\sim  \frac{\sigma\, T}{(\theta_0+\beta_0)}\sum_{k=1}^N \lambda_k^{p-1/2},\quad \mbox{as}\ T\to \infty\nonumber\\
 & \sim \frac{\sigma\, \varpi dTN^{\frac{p-2}{d}+1}}{\left(4+2d\right)\, (\theta_0+\beta_0)},\quad \mbox{as}\ N,T\to \infty. 
\end{align}

For $N,T$ define 
$$
Z_{N,T}:=  \frac{I_{2,2}}{ 
  \sqrt{T}\,  N^{1/d\, + 1/2 }\, 
 \sqrt{var( \cj_{1,1})}}
$$
Thus, we have
\begin{align*}
 \E\big( Z_{N,T} \big) =  \frac{\E(I_{2,2})}{ 
  \sqrt{T}\,  N^{1/d\, + 1/2 }\, 
 \sqrt{\E(I_{2,2})}}  = \frac{\sqrt{\E(I_{2,2})}}{ 
  \sqrt{T}\,  N^{1/d\, + 1/2 }},
\end{align*}
and, by using \eqref{tech_1}, we get 
\begin{align*}
 \E\big( Z_{N,T} \big) \sim \frac{\sigma\, \varpi d}{\left(4+2d\right)\, (\theta_0+\beta_0)},\quad \mbox{as}\ N,T\to \infty. 
\end{align*}
Thus, the $Z_{N,T}$ converges in $L^1(\Omega)$ to the constant appearing in the last expression, and therefore, converges in probability to the same constant.
This shows the first convergence. The second one could be obtained in a similar manner.
\end{proof}

\begin{lem}\label{dm}
Consider the Hilbert space $\mathcal{H}:=L^2([0,T]\times \mathcal{M})$ where $\mathcal{M}$ is the counting measure on $\mathbb{N}$, then
\begin{align*}
\sqrt{Var\left(\frac{1}{2}\left\|\sqrt{var(\cj_{p,1})}DM_{N,T}\right\|^2_{\mathcal{H}}\right)}\to 0, \quad \text{as }\; N, T \to \infty,
\end{align*}
where $D$ is the Malliavin derivative defined in Definition C.\ref{malliavin_der} .
\end{lem}

The proof of Lemma \ref{dm} is similar to the proof of Lemma 1 in Ref.~\refcite{CDK}; and it is based on some Malliavin calculations. The main idea is that, after to get the following inequality
\begin{equation*}
\sqrt{Var\left(\frac{1}{2}\left\|\sqrt{var(\cj_{p,1})}DM_{N,T}\right\|^2_{\mathcal{H}}\right)}\leq \frac{1}{2}(B_1+B_2+B_3),
\end{equation*}
where
\begin{align*}
B_1:=&\int_0^T\left(Var\left[\sum_{k=1}^{N}\lambda_k^{2p+2\gamma}u_k^2(r)\right]\right)^{1/2}dr\\
B_2:=&2\sigma \int_0^T\left(Var\left[\sum_{k=1}^{N}\lambda_k^{2p+\gamma}u_k(r)\int_r^Te^{(-\theta_0\lambda_k-\beta_0\lambda_k^{1/2})(t-r)}dw_k(t)\right]\right)^{1/2}dr\\
B_3:=&\sigma^2\int_0^T\left(Var\left[\sum_{k=1}^{N}\lambda_k^{2p}\left(\int_r^Te^{(-\theta_0\lambda_k-\beta_0\lambda_k^{1/2})(t-r)}dw_k(t)\right)^2\right]\right)^{1/2}dr,
\end{align*}
use the properties of the OU process \eqref{eq:OU-Fourier}, such as the fourth moments, independence, etc., to proof that $B_i\to 0$ as $N, T \to \infty$ for $i=1,2,3$. In this way, we conclude the asymptotic normality applying Theorem B.\ref{MS-con}.

The idea of the proof for $\beta$ is very similar.

%%%%%%%%%%%%%%%%%%%%%%%%%%%%%%%%%%%%%%%%%%%%%%%%%%%%%%%%%%%%%%%%%%%%%%%%%%%%%%%%%%%%%%%%%%%%%%%%%%%%%%%%%%%%%%%%%%%%%%%%%%%%%%%%%%%%%%%%%%%%%%%%%%%%%%%%%%%%%%%%%%%%%%%%%%%%%%%%%%%%%%%%%%%%%%%%%%%%%%%%%%%%%%%%%%%%%%%%%%%%%%
\section{Simulation study}\label{simu}
In this section, we present a simulation study to illustrate the statistical properties and validate the MLEs. 

First, we note that from \eqref{eq:MLE-theta}, and \eqref{eq:MLE-beta} and after some elemental algebra we obtain
\begin{align*}
\widehat{\theta}_{N,T} &= \frac{1}{I_{2,2} I_{1,2} - I_{3/2,2}^2} \big( I_{3/2,2} J_{1/2,2} - I_{1,2} J_{1,2}\big),\\    
\widehat{\beta}_{N,T} &= \frac{1}{I_{2,2} I_{1,2} - I_{3/2,2}^2} \big( I_{3/2,2} J_{1,2} - I_{2,2} J_{1/2,2}\big), 
\end{align*}

These last expressions provide a version of the estimators that we use to perform the numerical experiments of this work.

\subsection{Numerical validation of the MLEs}
Here, we present a practical and simple case that will serve to study the validation of the MLEs for the two parameters $\beta$ and $\theta$.\\

Consider the SPDE \eqref{stochastic_local_nonlocal} where for simplicity $d=1$ and $\mathcal{O}=[0,1]$. 

The complete orthonormal system of eigenfunctions $h_k$ for $L^2(\mathcal{O})$ is defined as
\begin{align*}
h_k(x)=
\begin{cases}
  1,    & k=1\\
  \sqrt{2}cos((k-1)\pi x),    & k=2,\ldots.
\end{cases}
\end{align*}

For simplicity, we are taking $U_0(x)\equiv 0$, which implies that $u_k(0)=0$ for all $k\ge 0$.\\

At this point, we assume that the parameters  $\beta$ and $\theta$ are known and we use simulations of the solution of the Ornstein-Uhlenbeck differential equation \eqref{eq:OU-Fourier}. In addition, we are using $\gamma=0$. We generate paths of OU processes in the interval time $[0,1]$ using Milstein Scheme with a discretization of $\Delta=0.001$. The parameters used for the simulations were $\theta=0.5,\beta=10.0,\sigma=1.0$, and $x=0.1$.  Thus, using the formulas in Equations \eqref{eq:MLE-theta} and \eqref{eq:MLE-beta} we obtained the MLEs of $\theta$ and $\beta$ based on data given by the paths generated. This gives us the estimated parameters  $\widehat{\beta}_{N,T}$ and $\widehat{\theta}_{N,T}$ and compares them with the true values. \\

Figure \ref{fig:conv_time} illustrates numerical evidence of the consistency for the MLEs provided in Theorem \ref{theo:cons}. Here, we are considering a fixed $N=50$ and the time parameter is varying in the interval $[0,1]$, i.e. $T=i\Delta$ for $i=10,20,\ldots,1,000$. From the numerical evidence, we can say that when $T=0.3$ both estimators have converged to the true value. The value of $N=50$ was chosen as the smallest value that ensures a good convergence of the MLEs, as $T$ increases. That means that the numerical approximation \eqref{u_truncation} is good enough to provide a small error on the approximation as well for the estimators.

\begin{figure}[H]
\centering
%\subfloat[An example of an individual figure sub-caption.]{%
\resizebox*{10.5cm}{!}{\includegraphics[width=25cm, height=20cm]{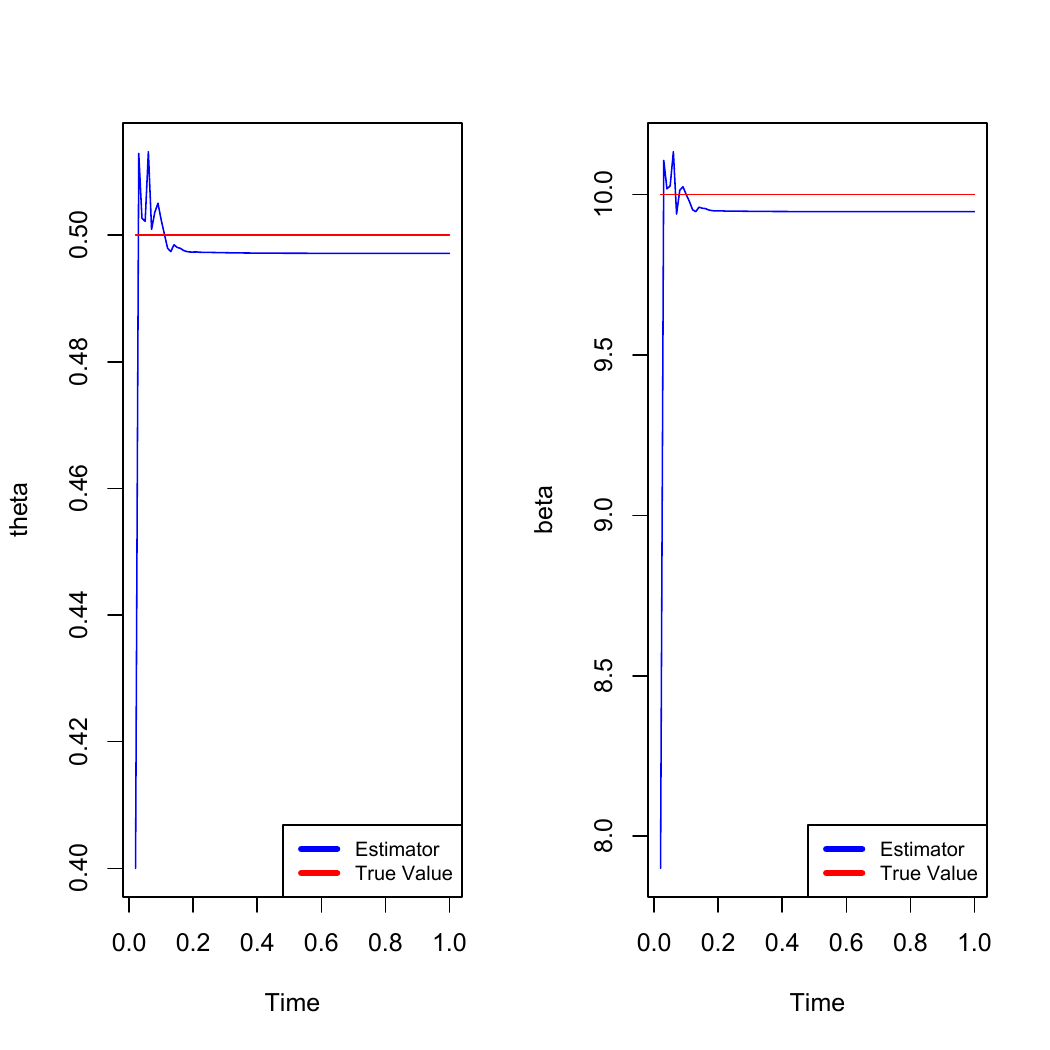}}
\caption{Consistency of MLEs when the parameters are $\beta=10$ and $\theta=0.5$ and $T$ grows.} \label{fig:conv_time}
\end{figure}

On the other side, knowing that convergence is guaranteed when $T=1$, Figure \ref{fig:conv_N} exemplifies the convergence when $N$ increases from $1$ to $100$. We can observe that for $N=50$, we have convergence for both estimators.

\begin{figure}[H]
\centering
%\subfloat[An example of an individual figure sub-caption.]{%
\resizebox*{10.5cm}{!}{\includegraphics[width=25cm, height=20cm]{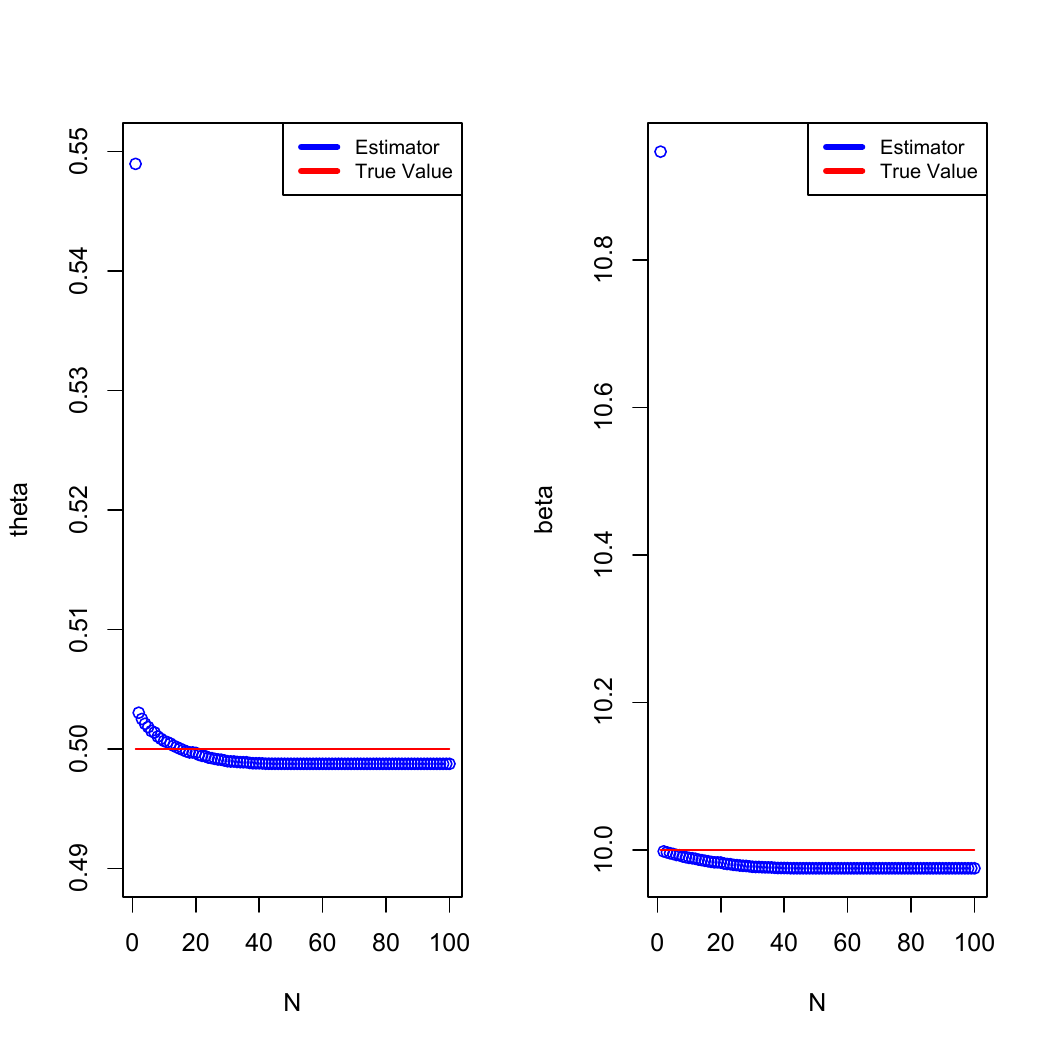}} 
\caption{Convergence of parameters with $\beta=10$ and $\theta=0.5$ when $N$ grows.} \label{fig:conv_N}
\end{figure}
Finally, we simulated $1,000$ data set paths in the time interval $[0,1]$ and $N=50$ to calculate the mean and quantiles ( 95\%). These are reported in Table \ref{tab:exa1}.

\begin{table}[H]
\begin{center}
\begin{minipage}{350pt}
\caption{Average (estimator) and quantiles ($95\%$) of parameter estimates.}
\label{tab:exa1}
\end{minipage}

\hspace*{-2.5cm}
\begin{minipage}{178pt}
\begin{tabular}{c| c| c | c}
\textbf{Parameter} &\textbf{Real value} &\textbf{Estimator} & \textbf{Quantile 95\%}\\
\hline
\hline
$\theta$ & 0.5 & 0.49816   & (0.48556,0.51204) \\
\hline
$\beta$ & 10& 9.93528 & (9.89349,10.07518) \\
\hline
\end{tabular}

\end{minipage}
\end{center}
\end{table}

%\begin{figure}[h]
%\centering
%\subfloat[An example of an individual figure sub-caption.]{%
%\resizebox*{7.5cm}{!}{\includegraphics{1_theta_0_85}} \hspace{5pt}
%\subfloat[A slightly shorter sub-caption.]{%
%\resizebox*{7.5cm}{!}{\includegraphics{1_beta_0_45}}
%\caption{Estimation of parameters with $\beta=0.45$ and $\theta=0.85$.} \label{MLE_validation1}
%\end{figure}

%\begin{figure}[h]
%\centering
%\subfloat[An example of an individual figure sub-caption.]{%
%\resizebox*{7.5cm}{!}{\includegraphics{2_theta_1_1}} \hspace{5pt}
%\subfloat[A slightly shorter sub-caption.]{%
%\resizebox*{7.5cm}{!}{\includegraphics{2_beta_0_88}}
%\caption{Estimation of parameters with $\beta=0.88$ and $\theta=1.1$.} \label{MLE_validation2}
%\end{figure}

\subsection{Simulations of solutions of the SPDE}

This section is devoted to presenting simulations of one example of SPDE in one dimension. We generate paths of OU processes in the interval time $[0,1]$ using Milstein Scheme with a discretization of $\Delta=0.001$. The parameters used for the  simulations were $\theta=2.0,\beta=1.0,\sigma=0.1$, $\gamma=0$, $N=100$  and  with a non-zero initial condition given $u_k(0)=\big(u(0,\cdot), h_k(\cdot)\big)_{L^2(\mathcal{O})}$, where  the initial condition was set as  $u(0,x)= (6/7)\,(1+x-x^2)$  for $x\in [0,1]$.\\

Figure \ref{fig:mov_paths} shows a simulation of the stochastic density using the SPDE \eqref{stochastic_local_nonlocal}. The simulation to $u$ has been done using the numerical approximation $u^N$ given by \eqref{u_truncation}, with $N=100$ and a discretization for $x\in [0,1]$ of $0.001$.

\begin{figure}[H]
\centering
%\subfloat[An example of an individual figure sub-caption.]{%
\resizebox*{10.5cm}{!}{\includegraphics[width=25cm, height=20cm]{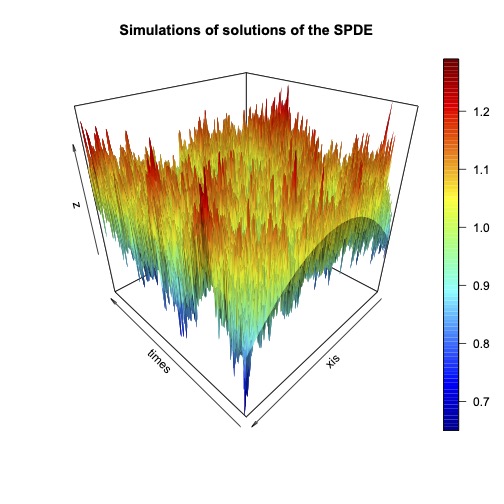}} 
\caption{Stochastic density of animal movements, represented by the SPDE \eqref{stochastic_local_nonlocal} and using a numerical approximation of the solution $u$.} \label{fig:mov_paths}
\end{figure}

\section{Concluding remarks}\label{conclu}

In this work, we have used a stochastic partial differential equation to model ecological niches with some particular properties; here we have considered the sum of classical and fractional Laplacians. This model is biologically relevant in situations of a population following long-jump foraging patterns alternated with focused searching strategies at small scales (see Ref.~\refcite{Dipi-valdi}). We applied the Galerkin projection to the SPDE and we see that every coefficient of the spectral decomposition satisfy an SDE: the well-known OU process. \\

We observe that it is possible to approximate the density $u(t,\xi)$ at every time $t\in [0,T]$ and space point $\xi$ by using its truncation $u^N(t,\xi)$ (see \eqref{u_spectral} and \eqref{u_truncation}), thus, we could use this approximation to simulate the SPDE's solution and to estimate the parameters. Indeed, we have performed MLE for this finite-dimensional Galerkin approximation. We have proved the consistency of the MLE of the parameters, and we sketched the proof of the asymptotic normality using the Malliavin-Stein method. We also have presented numerical experiments to validate the statistical method. The numerical experiments confirm that the MLE converges quickly, when $N,T \rightarrow \infty$, to the true parameter. \\

As a future work, it could be possible to consider an affine-type noise, instead of the additive noise, in such a manner that the SPDE still could be diagonalizable; so the method used in this work still is applicable. We could consider exploring the model using an alternative definition of the fractional Laplacian. However, the current inference method is no longer suitable, and we must adopt a completely different approach.\\

Given that both a jump process and a random walk are present, we can apply this model and method to finance as well. Certainly, this type of model provides a robust framework for modeling financial assets. Incorporating jump processes into SPDEs, as elaborated in Ref.~\refcite{Cont}, allows for a more precise depiction of market fluctuations and extreme events. This presents an invaluable tool for thorough analysis and accurate prediction of asset price dynamics.

%%%%%%%%%%%%%%%%%%%%%%%%%%%%%%%%%%%%%%%%%%%%%%%
%%%%%%%%%%%%%%%%%%%%%%%%%%%%%%%%%%%%%%%%%%%%
\appendix

\section{Some useful results}

For completeness, we present the following version of the classical Strong Law of Large Numbers (see Theorem 2.3.2  in Ref.~\refcite{pavon}). 

\begin{theorem}[Strong Law of Large Numbers]
\label{SLLN}
Let $\xi_k, \ k\geq 1,$ be independent random variables with the following
properties:
\begin{itemize}
\item $\E (\xi_k)=0$,\ $\E( \xi_k^2)>0$,
\item There exist real numbers $c>0$ and $ \alpha\geq -1$ such that
$$
\lim_{k\to \infty} k^{-\alpha}\E (\xi_k^2)=c.
$$
\end{itemize}
Then, with probability one,
$$
\lim_{N\to \infty} \frac{\sum_{k=1}^N\xi_k}{\sum_{k=1}^N \E (\xi_k^2)}=0.
$$
If, in addition, $\E (\xi_k^4)\leq c_1\Big(\E(\xi_k^2)\Big)^2$ for all $k\geq 1$,
with $c_1>0$ independent of $k$, then, also with probability one,
$$
\lim_{N\to \infty} \frac{\sum_{k=1}^N \xi_k^2}{\sum_{k= 1}^N \E (\xi_k^2)}=1.
$$
\end{theorem}

\section{Malliavin-Stein method}

In this subsection, we introduce the framework of Malliavin calculus and its connection with  Stein's method. We present some necessary stools and definitions to achieve the asymptotic normality of $\hat{\theta}_{N,T}$ and $\hat{\beta}_{N,T}$ using the Malliavin-Stein's method.

Let $T>0$ be given. We will suppose that $\mathcal{H}$ is a separable Hilbert Space of the form $\mathcal{H}=L^2([0,T]\times \mathcal{M})$ where $\mathcal{M}$ is the counting measure on $\mathbb{N}$, more precisely, $v\in\mathcal{H}$ is of the form
\begin{equation*}
v(t)=\sum_{k=1}^{\infty}v_k(t),
\end{equation*}
and we endow $\mathcal{H}$ with the norm
\begin{equation*}\label{df:norm}
\|v\|_{\mathcal{H}}:=\left(\sum_{k=1}^{\infty}\int_0^T|v_k(t)|^2dt\right)^{1/2}.
\end{equation*}

Let $W=\{W(h): h\in \mathcal{H}\}$ be an {\it isonormal Gaussian process} associated with $\mathcal{H}$ defined on a complete probability space $(\Omega, \mathbb{F}, \mathcal{P})$ where, without loss of generality, we can assume that $\mathbb{F}=\sigma(W)$ is the $\sigma-$algebra generated by $W$. In other words, an isonormal Gaussian process is simply a family of centered Gaussian random variables with the covariance structure given by $\mathbb{E}(W(h)W(g))=\langle h,g\rangle_{\mathcal{H}}$. 

For $q\geq 1$, consider the $q$th Hermite polynomial defined as
\begin{equation*}
H_q(x)=(-1)^qe^{\frac{x^2}{2}}\frac{d^q}{dx^q}(e^{-\frac{x^2}{2}}).
\end{equation*}
We denote by $\mathfrak{H}_q$ to the closed linear subspace of $L^2(\Omega):=L^2(\Omega, \mathbb{F}, \mathcal{P})$ generated by $\{H_q(W(h)):h\in \mathcal{H},\;\|h\|_{\mathcal{H}}=1\}$ which is known as the $q$th Wiener chaos of $W$. The space $L^2(\Omega)$ can be decomposed in the orthogonal sum of the spaces $\mathfrak{H}_q$, therefore, any square integrable random variable $F\in L^2(\Omega)$ admits the following Wiener-It\^o chaotic expansion 
\begin{equation}\label{def:chaotic}
F=\sum_{q=0}^{\infty}I_q(f_q),
\end{equation}
where $I_q$ is the multiple integral of order $q$ defined as 
\begin{equation*}
I_q(f_q)=q!\int_0^TdW(t_1)\int_0^{t_1}dW(t_2)\cdots \int_0^{t_{q-1}}dW(t_q)f_q(t_1,\ldots , t_q),\; f_	q\in L^2([0,T]^q),\;
\end{equation*}
with $0<t_1<t_2<\cdots t_q <T$.

Now, we present some basic elements of the Malliavin calculus with respect to the isonormal Gaussian process W. Denote by $C_p^{\infty}(\mathbb{R}^n)$ the set of all infinitely continuously differentiable functions $f:\mathbb{R}^n\to \mathbb{R}$ such that $f$ and all its partial derivatives have polynomial growth. Let $\mathcal{S}$ be the class of smooth random variables $F$ of the form 
\begin{equation}\label{smth}
F=f(W(h_1),\ldots, W(h_n)),
\end{equation}
where $f\in C_p^{\infty}(\mathbb{R}^n)$ and $h_i \in \mathcal{H}$ for all $1\leq i\leq n $. The derivative of this class of random variables is the following.
\begin{defi}\label{malliavin_der}
The Malliavin derivative of a smooth random variable $F\in\mathcal{S}$ of the form \eqref{smth} is given by
\begin{equation*}
DF=\sum_{i=1}^n\frac{\partial f}{\partial x_i}(W(h_1),\ldots, W(h_n))h_i.
\end{equation*}
\end{defi}
We remark that the operator $D:\mathcal{S}\subset L^p(\Omega)\to L^p(\Omega;\mathcal{H})$ is closable for any $p\geq 1.$ Therefore, we denote by $\mathbb{D}^{1,p}$ the closure of $\mathcal{S}$ we respect to the norm
\begin{equation*}
\|F\|_{1,p}=\left[\mathbb{E}\left(|F|^p\right)+\mathbb{E}\left(\|DF\|^p_{\mathcal{H}}\right)\right]^{1/p}.
\end{equation*}

The Malliavin derivative satisfies the following chain rule.
\begin{prop}
Let $\psi:\mathbb{R}^{n}\to \mathbb{R}$ be a continuously differential function with bounded partial derivatives. If $F=(F_1,\ldots , F_n)$ is a random vector whose components belong to the space $\mathbb{D}^{1,2}$. Then $\phi(F)\in \mathbb{D}^{1,2}$ and 
\begin{equation}\label{chainrule}
D(\psi(F))=\sum_{i=1}^{n}\frac{\partial \psi}{\partial x_i}(F)DF_{i}.
\end{equation}
\end{prop}

In summary, the Malliavin derivative $D$ is a closed and unbounded operator that takes values in $L^2([0,T]\times \Omega)$ and it is defined on the dense subset $\mathbb{D}^{1,p}$ of $L^2(\Omega)$. 

The following result Ref.~\refcite[Proposition 1.2.1]{Nualart} establishes under which conditions the random variable \eqref{def:chaotic} belongs to the domain of the derivation operation and how we compute its Malliavin derivative.

\begin{prop}
Let $F$ be a random variable defined as in \eqref{def:chaotic}. Then $F\in \mathbb{D}^{1,2}$ if and only if 
\begin{equation*}
\sum_{q=1}^{\infty}qq!\|f_q\|^2_{L^2([0,T]^q)}<\infty
\end{equation*}
and in this case, we have
\begin{equation*}
D_tF=\sum_{q=1}^{\infty}qI_{q-1}(f_q(\cdot, t)).
\end{equation*}
\end{prop}

The operator $L$, known as the infinitesimal generator of the Ornstein-Uhlenbeck semigroup, is defined as 
\begin{equation*}
L=\sum_{q=0}^{\infty}-qL_q, 
\end{equation*}
where $J_q$ denotes the projection operator onto $\mathfrak{H}_q$ and its domain is formed by the random variables $F\in L^2(\Omega)$ such that $\sum_{q=1}^{\infty}q^2\|J_qF\|^2_{L^2(\Omega)}<\infty$. The pseudo-inverse operator $L^{-1}$ is defined as $L^{-1}F=\sum_{q=1}^{\infty}-\frac{1}{q}J_q(F)$, for any $F\in L^2(\Omega)$, in addition, $L^{-1}F\in \text{Dom}L$ and $LL^{-1}F=F-\mathbb{E}(F)$.

We also introduce the following infinite dimensional Malliavin integration by parts formula Ref.~\refcite[Theorem 2.9.1]{NorPec}.%Normal Approximations with Malliavin Calculus: From Stein’s Method to Universality

\begin{lem}
Let $F\in \mathbb{D}^{1,2}$ and $G\in L^2(\Omega)$. Then $L^{-1}G \in \text{Dom}L$ and 
\begin{equation}\label{ibp}
\mathbb{E}(FG)=\mathbb{E}(F)\mathbb{E}(G)+\mathbb{E}\left(\langle DF,-DL^{-1}G\rangle_{\mathcal{H}}\right).
\end{equation}
\end{lem}

After this short introduction to some basic concepts of Malliavin calculus, we present its connection with Stein's method which allows us to obtain bounds between two probability distributions with respect to a metric. First, we present the following lemma.

\begin{lem}\label{SL}
Let $Y$ be an integrable random variable. Then, $Y$ is a standard normal random variable if and only if $\mathbb{E}\left(f^{\prime}(Y)\right)=\mathbb{E}\left(Yf(Y)\right)$ for any function $f:\mathbb{R}\to \mathbb{R}$ continuously differentiable such that $f^{\prime}(Y)$, $Yf(Y)\in L^{1}(\Omega)$ and $f$, $f^{\prime}$ have at most polynomial growth at infinity.
\end{lem}

Roughly speaking, the Stein's Lemma \ref{SL} hints that a real random variable $Y$ is close to the standard normal distribution whenever $\mathbb{E}\left(f^{\prime}(Y)\right)-\mathbb{E}\left(Yf(Y)\right)$ is closed to zero for any $f$ in a deterministic class of functions. 

Now, consider a centered random variable $F\in \mathbb{D}^{1,2}$ with $\mathbb{E}(F^2)=1$ and $f\in C^1$ such that $\|f\|<c$ and $\|f^{\prime}\|<2$. Using the integration by parts formula \eqref{ibp} and the chain rule \eqref{chainrule} we have that
\begin{align*}
\left|\mathbb{E}\left(f^{\prime}(F)-Ff(F)\right)\right|&=\left|\mathbb{E}\left(f^{\prime}(F)\left[1-\langle DF,-DL^{-1}F\rangle_{\mathcal{H}}\right]\right)\right|\\
&\leq 2 \mathbb{E}\left(\left| 1-\langle DF,-DL^{-1}F\rangle_{\mathcal{H}}\right|\right)
\end{align*}
and if we assume that $F\in\mathbb{D}^{1,4}$, then $1-\langle DF,-DL^{-1}F\rangle_{\mathcal{H}}$ is square integrable, therefore by Cauchy-Schwarz inequality we obtain 
\begin{align*}
\left|\mathbb{E}\left(f^{\prime}(F)-Ff(F)\right)\right|\leq 2 \sqrt{\text{Var}\left(\langle DF,-DL^{-1}F\rangle_{\mathcal{H}}\right)}
\end{align*}
The above arguments allow us to introduce the following results Ref.~\refcite{nordin-peccati}. 
%Stein’s method on Wiener chaos and el libro de los ivanes y pecosos

\begin{theo}
Let $q\geq 2$. If $F=I_q(f)$ belongs of Wiener chaos of order $q$ such that $\mathbb{E}(F^2)=\sigma^2>0$ and $N\sim \mathcal{N}(0,\sigma^2)$. Then, the total variation of the random variables $F$ and $N$ satisfies that
\begin{equation*}
\text{d}_{\text{TV}}(F,N)\leq \frac{2}{\sigma^2}\sqrt{\text{Var}\left(\langle DF, -DL^{-1}F\rangle_{\mathcal{H}}\right)}=\frac{2}{\sigma^2}\sqrt{\text{Var}\left(\frac{1}{q}\|DF\|^2_{\mathcal{H}}\right)}
\end{equation*}
\end{theo}

\begin{theo}\label{MS-con}
Let $F_n=I_q(f_n)$ with $n\geq 1$ a sequence of random variables with $q\geq 2$. Assume that $\mathbb{E}(F^2_n)\to \sigma^2$. Then the following assertions are equivalents 
\begin{enumerate}
  \item $F_n\rightarrow N$ in law,
  \item $\text{d}_{\text{TV}}(F_n,N)\to 0$.
\end{enumerate}
\end{theo}
%%%%%%%%%%%%%%%%%%%%%%%%%%%%%%%%%%%%%%%%%%%%%%%%%%%%%%%%%%%%%%%%%%%%%%%%%%%%%%%%%%%%%%%%%%%%%%%%%%%%%%%%%%%%%%%%%%%%%%%%%%%%%%%%%%%%%%%

\section*{Acknowledgment}
%This section should come after the Appendices if any and should
%be unnumbered. Funding information may also be included here.
Baltazar-Larios F. has been supported by UNAM-DGAPA-PAPIIT-IN102224. The work of L. Peralta is supported by UNAM-DGAPA PAPIIT grants IA100324 (Mexico).

% ORCIDs 
% Fer https://orcid.org/my-orcid?orcid=0000-0001-8292-3615
%Fran
%Lili  https://orcid.org/0000-0002-8611-1734 


\begin{thebibliography}{00}
\bibitem{Alt-22}
Altmeyer, R., Bretschneider, T., Janák, J., \& Reiß, M. (2022). \newblock {\em Parameter estimation in an SPDE model for cell repolarization.}
\newblock SIAM/ASA Journal on Uncertainty Quantification.

\bibitem{aoki} Aoki, I. (2012). Entropy principle for the development of complex biotic systems: organisms, ecosystems, the Earth. Elsevier.

\bibitem{casella-berger} Casella, G., \& Berger, R. L. (2001). Statistical inference. Cengage Learning.


\bibitem{chavanis2010stochastic}
Chavanis, P.-H. (2010).
\newblock A stochastic keller--segel model of chemotaxis.
\newblock {\em Communications in Nonlinear Science and Numerical Simulation},
  15(1):60--70.


\bibitem{cia-18} Cialenco, I. (2018). Statistical inference for SPDEs: an overview, Stat Inf Stoch Proc 21, 309–329.

\bibitem{CDK}
Cialenco, I., Delgado-Vences, F., and Kim, H.-J. (2020).
\newblock Drift estimation for discretely sampled spdes.
\newblock {\em Stochastics and Partial Differential Equations: Analysis and
  Computations}, pages 1--26.

 \bibitem{Cont} 
Cont, R. and Tankov, P..
\newblock {\em Financial Modelling with Jump Processes }. 
 Chapman and Hall/CRC, New York.2004.


\bibitem{Dipi-Pro}{Dipierro, S., Proietti, E. and Valdinoci, E. (2022). (Non)local logistic equations with Neumann conditions. Ann. Inst. H. Poincaré Anal. Non Linéaire}

\bibitem{Dipi-valdi} Dipierro, S., \& Valdinoci, E. (2021). Description of an ecological niche for a mixed local/nonlocal dispersal: an evolution equation and a new Neumann condition arising from the superposition of Brownian and Lévy processes. Physica A: Statistical Mechanics and its Applications, 575, 126052.




\bibitem{dogan2016derivation}
Dogan-Ciftci, E. and Bulut, U. (2016).
\newblock Derivation of stochastic biased and correlated random walk models.
\newblock {\em Neural, Parallel, and Scientific Computations}, 24:163--180.


%\bibitem{fl-hu} Flandoli, F., \& Huang, R. (2021). The KPP equation as a scaling limit of locally interacting Brownian particles. Journal of Differential Equations, 303, 608-644.
%
%\bibitem{fl-hu-2} Flandoli, F., \& Huang, R. (2022). Coagulation dynamics under environmental noise: Scaling limit to SPDE. ALEA, 19, 1241-1292.

\bibitem{gloaguen2018stochastic}
Gloaguen, P., Etienne, M.-P., and Le~Corff, S. (2018).
\newblock Stochastic differential equation based on a multimodal potential to
  model movement data in ecology.
\newblock {\em Journal of the Royal Statistical Society: Series C (Applied
  Statistics)}, 67(3):599--619.


\bibitem{gr-ng} Grebenkov, D. S., \& Nguyen, B. T. (2013). Geometrical structure of Laplacian eigenfunctions. \newblock {\em siam REVIEW}, 55(4), 601-667.



\bibitem{guo-yu} Guo, Qinghua, and Yu Liu. "ModEco: an integrated software package for ecological niche modeling." Ecography 33.4 (2010): 637-642.

\bibitem{har} Harte, John Maximum entropy and ecology: a theory of abundance, distribution,
and energetics. OUP Oxford, 2011.

\bibitem{hil-21} Hildebrandt, F., \&  Trabs, M. (2021). Parameter estimation for SPDEs based on discrete observations in time and space. Electron. J. Statist. 15 (1) 2716-2776.

\bibitem{hue-93} Huebner, M.,Khasminskii, R. \& Rozovskii , B. L. (1993).  Two Examples of Parameter Estimation for Stochastic Partial Differential Equations. In: Cambanis, Ghosh, Karandikar, Sen (eds.) Stochastic Processes, Springer.

\bibitem{hue-95} Huebner, M. \& Rozovskii , B. L. (1995).  On asymptotic properties of maximum likelihood estimators for parabolic stochastic PDE’s. Probability Theory and Related Fields 103, 143–163

\bibitem{iacus2008simulation}
Iacus, S.~M. (2008).
\newblock {\em Simulation and inference for stochastic differential equations:
  with R examples}, volume 486.
\newblock Springer.

%\bibitem{ja-wa} Jabin, P. E., \& Wang, Z. (2017). Mean field limit for stochastic particle systems. Active Particles, Volume 1: Advances in Theory, Models, and Applications, 379-402.

\bibitem{jay-57-1} Jaynes, E. T. (1957). \newblock {\em  Information theory and statistical mechanics}. Physical review, 106(4), 620.


\bibitem{jay-57-2} Jaynes, E. T. (1957). \newblock {\em  Information theory and statistical mechanics. II.} Physical review, 108(2), 171.

\bibitem{jor} Jorgensen, Sven E Thermodynamics and ecological modelling. CRC press, 2018.

%\bibitem{ki-la} Kipnis, C., \& Landim, C. (1998). Scaling limits of interacting particle systems (Vol. 320). Springer Science \& Business Media.

\bibitem{Loto2003}
Lototsky, S. (2003). 
\newblock {\em Parameter estimation for stochastic parabolic equations: asymptotic properties of a two-dimensional projection-based estimator.}
\newblock Statistical inference for stochastic processes, 6, 65--87.

\bibitem{LototskyRozovsky2017Book}
Lototsky, S.~V. and Rozovsky, B.~L. (2017a).
\newblock {\em Stochastic partial differential equations}.
\newblock Universitext. Springer International Publishing.

\bibitem{monte}{Montefusco, E.  and Pellacci, B., and Verzini, G. (2013). Fractional diffusion with Neumann boundary conditions: the logistic equation. Discrete Contin. Dyn. Syst. Ser. B 18, no. 8, 2175--2202. MR3082317}

\bibitem{NorPec}{Nourdin, I., Peccati, G. (2012). Normal approximations with Malliavin calculus: from Stein's method to universality (Vol. 192). Cambridge University Press.}

\bibitem{nordin-peccati}{ Nourdin, I., Peccati, G. (2009). Stein’s method on Wiener chaos. Probab. Theory Relat. Fields 145, 75–118 }

\bibitem{Nualart} {Nualart, D. (2006). The Malliavin calculus and related topics (Vol. 1995, p. 317). Berlin: Springer.}

\bibitem{oksendal2013stochastic}
Oksendal, B. (2013).
\newblock {\em Stochastic differential equations: an introduction with
  applications}.
\newblock Springer Science \& Business Media.

\bibitem{okubo2001diffusion}
Okubo, A. and Levin, S.~A. (2001).
\newblock {\em Diffusion and ecological problems: modern perspectives},
  volume~14.
\newblock Springer.

\bibitem{pavon} Pav\'on Español, José Julián (2021),  \newblock {\em The stochastic wave equation and its maximum likelihood estimators.} Master's Thesis. UNAM.

\bibitem{pra-00}
Rao B. P. (2000).
\newblock Bayes estimation for some stochastic partial differential equations.
\newblock {\em Journal Statistical planning inference.}, 91:511--524.

\bibitem{rao-01}
Rao, B. P. (2001).
\newblock Statistical inference for stochastic partial differential equations.
\newblock {\em Lecture Notes-Monograph Series}, 47--70.

\bibitem{rs}{Reed, M., and Simon, B., Methods of modern mathematical physics. IV. Analysis of operators. Academic Press [Harcourt Brace Jovanovich, Publishers], New York-London, 1978. {\rm xv}+396 pp.}

\bibitem{rey}Reynolds, A. M., \& Rhodes, C. J. (2009). 
\newblock The Lévy flight paradigm: random search patterns and mechanisms. 
\newblock {\em  Ecology, 90(4), 877-887.}


\bibitem{schweitzer1994clustering}
Schweitzer, F. and Schimansky-Geier, L. (1994).
\newblock Clustering of “active” walkers in a two-component system.
\newblock {\em Physica A: Statistical Mechanics and its Applications},
  206(3-4):359--379.

\bibitem{se-va-14} Servadei, R., \& Valdinoci, E. (2014). On the spectrum of two different fractional operators. Proceedings of the Royal Society of Edinburgh Section A: Mathematics, 144(4), 831-855. doi:10.1017/S0308210512001783

%\bibitem{Sh} Shiryaev, A.~N.
%\newblock {\em  Probability}, Graduate Texts in Mathematics, vol. 2, 2nd edn. 
%\newblock Springer, New York (1996)

\bibitem{smouse2010stochastic}
Smouse, P.~E., Focardi, S., Moorcroft, P.~R., Kie, J.~G., Forester, J.~D., and
  Morales, J.~M. (2010).
\newblock Stochastic modelling of animal movement.
\newblock {\em Philosophical Transactions of the Royal Society B: Biological
  Sciences}, 365(1550):2201--2211.

\bibitem{stevens2000derivation}
Stevens, A. (2000).
\newblock The derivation of chemotaxis equations as limit dynamics of
  moderately interacting stochastic many-particle systems.
\newblock {\em SIAM Journal on Applied Mathematics}, 61(1):183--212.

\end{thebibliography}
\end{document}